\documentclass[10pt, leqno]{article}
\usepackage[utf8]{inputenc}

\usepackage{amsmath,amssymb,amsthm, epsfig}

\usepackage{mathptmx}

\usepackage[usenames,dvipsnames]{pstricks}

\usepackage{pst-grad} 

\usepackage{pst-plot} 

\usepackage{stackrel}

\usepackage[nottoc,notlot,notlof]{tocbibind}

\usepackage{hyperref}

\usepackage{esint}

\usepackage{amsmath}

\usepackage{amssymb}

\usepackage{hyperref}

\usepackage{color}

\usepackage{ulem}

\usepackage{epsfig}

\usepackage[mathscr]{eucal}

\usepackage[usenames,dvipsnames]{pstricks}
 \usepackage{epsfig}
 \usepackage{pst-grad} 
 \usepackage{pst-plot} 

\title{Regularity up to the boundary for singularly perturbed fully nonlinear elliptic equations}
\author{\it by \smallskip \\
 G. C. Ricarte\footnote{\noindent \textsc{Gleydson Chaves Ricarte}.
Universidade Federal Cear\'{a} - UFC. Department of Mathematics. Fortaleza - CE, Brazil - 60455-760.
\texttt{E-mail address: ricarte@mat.ufc.br}}  \quad $\&$  \quad J. V. da Silva\footnote{\noindent \textsc{Jo\~{a}o V\'{i}tor da Silva}. Universidad de Buenos Aires. FCEyN, Department of Mathematics. Buenos Aires, Argentina. \noindent \texttt{E-mail address:  jdasilva@dm.uba.ar}}
}

\date{}

\newlength{\hchng}
\newlength{\vchng}
\def \R {\mathbb{R}}

\def \suchthat {\ \big | \ }

\newcommand{\pe}{E_{\varepsilon}}
\newcommand{\defeq}{\mathrel{\mathop:}=}

\newtheorem{theorem}{Theorem}[section]

\newtheorem{lemma}[theorem]{Lemma}

\theoremstyle{definition}

\newtheorem{definition}[theorem]{Definition}

\newtheorem{example}[theorem]{Example}

\theoremstyle{remark}

\newtheorem{remark}[theorem]{Remark}

\numberwithin{equation}{section}

\newcommand{\intav}[1]{\mathchoice {\mathop{\vrule width 6pt height 3 pt depth  -2.5pt
\kern -8pt \intop}\nolimits_{\kern -6pt#1}} {\mathop{\vrule width
5pt height 3  pt depth -2.6pt \kern -6pt \intop}\nolimits_{#1}}
{\mathop{\vrule width 5pt height 3 pt depth -2.6pt \kern -6pt
\intop}\nolimits_{#1}} {\mathop{\vrule width 5pt height 3 pt depth
-2.6pt \kern -6pt \intop}\nolimits_{#1}}}

\begin{document}
\maketitle

\begin{abstract}
In this article we are interested in studying regularity up to the boundary for one-phase singularly perturbed fully nonlinear elliptic problems, associated to high energy activation potentials, namely
$$
    F(X, \nabla u^{\varepsilon}, D^2 u^{\varepsilon}) = \zeta_{\epsilon}(u^{\epsilon})
    \quad \mbox{in} \quad \Omega \subset \R^n
$$
where $\zeta_{\varepsilon}$ behaves asymptotically as the Dirac measure $\delta_{0}$ as $\varepsilon$ goes to zero. We shall establish global gradient bounds independent of the parameter $\varepsilon$.
\medskip

\medskip



\end{abstract}
\tableofcontents

\newpage

\section{Introduction}\label{int}
\hspace{0.6cm}Throughout the last three decades or so, variational problems involving singular PDEs has received a warm attention as they often come from the theory of critical points of non-differentiable functionals.  The pioneering work of Alt-Caffarelli \cite{AC} marks the beginning of such a theory by carrying out the variational analysis of the minimization problem
\begin{equation}\tag{{\bf Minimum}}
   \displaystyle \min_{} \int_{\Omega} \left( |\nabla v|^2 + \chi_{\{v>0\}} \right) \ dX,
\end{equation}
among competing functions with the same non-negative Dirichlet boundary condition.

Since the very beginning it has been well established that such discontinuous minimization problems could be treated by penalization methods. Indeed, Lewy-Stampacchia, Kinderlehrer-Nirenberg, Caffarelli among others were the precursors of such an approach to the study of problem $\Delta u^{\epsilon} = \zeta_{\varepsilon}(u^{\varepsilon})$ over of 70s and 80s. Linear problems in non-divergence form was firstly considered by Berestycki \textit{et al} in \cite{BCN}.  Teixeira in \cite{Tei} started the journey of investigation into fully nonlinear elliptic equations via singular perturbation methods:
\begin{equation}\label{eqRT}
    F(X, D^2 u^{\varepsilon}) = \zeta_{\varepsilon}(u^{\varepsilon}) \quad \mbox{in} \quad \Omega,
\end{equation}
where $\zeta_{\varepsilon} \sim \varepsilon^{-1}\chi_{(0,\varepsilon)}$. The problem appears in nonlinear formulations of high energy activation models, see \cite{RT} and \cite{Tei}. It can also be employed in the analysis of overdetermined problems as follows. Given $\Omega \subset \mathbb{R}^n$ a domain and a non-negative function
$\varphi\colon  \Omega \rightarrow \mathbb{R}$, it plays a crucial role in Geometry and Mathematical Physics the question of finding a compact hyper-surface
$ \partial \Omega' \subset \Omega$ such that the following elliptic boundary value problem
\begin{equation}  \label{Free}
\left\{
\begin{array}{rclcl}
F(X, \nabla u, D^2 u)  &=&  0 & \mbox{in} & \Omega \backslash \Omega^{\prime}\\
u & = & \varphi & \mbox{on} & \partial \Omega\\
u & = & 0  & \mbox{in}  & \Omega^{\prime}\\
u_{\nu} & = & \psi &\mbox{in} & \partial \Omega^{\prime},
\end{array}
\right.
\end{equation}
can be solved. Problems as \eqref{eqRT} became known over the years in the Literature as \textit{cavity type problems}.

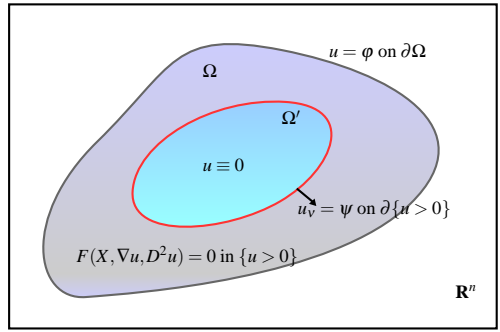
\begin{figure}[h]
\begin{center}
%
%
\psscalebox{0.7 0.7} 
{
\begin{pspicture}(0,-4.1)(9.3,4.1)
\definecolor{colour4}{rgb}{0.8,0.8,0.8}
\definecolor{colour3}{rgb}{0.8,0.8,1.0}
\definecolor{colour2}{rgb}{0.4,0.4,0.4}
\definecolor{colour6}{rgb}{0.6,1.0,1.0}
\definecolor{colour5}{rgb}{0.6,0.8,1.0}
\definecolor{colour1}{rgb}{1.0,0.2,0.2}
\psframe[linecolor=black, linewidth=0.04, dimen=outer](9.3,2.1)(0.0,-4.1)
\pscustom[linecolor=black, linewidth=0.04]
{
\newpath
\moveto(8.1,-1.1)
}

\pscustom[linecolor=black, linewidth=0.04]
{
\newpath
\moveto(9.3,-1.5)
}
\pscustom[linecolor=black, linewidth=0.04]
{
\newpath
\moveto(14.9,3.3)
}

\pscustom[linecolor=black, linewidth=0.04]
{
\newpath
\moveto(14.9,2.9)
}
\psbezier[linecolor=colour2, linewidth=0.04, fillstyle=gradient, gradlines=2000, gradbegin=colour3, gradend=colour4](1.6908844,-0.53186655)(3.1618953,0.90770125)(3.113053,1.622456)(5.204675,1.2099068)(7.2962976,0.7973577)(9.114757,-0.26643416)(7.6437454,-1.706002)(6.1727347,-3.1455698)(2.4471457,-3.409379)(1.4782214,-3.4827323)(0.5092973,-3.5560858)(0.21987355,-1.9714344)(1.6908844,-0.53186655)
\psbezier[linecolor=colour1, linewidth=0.04, fillstyle=gradient, gradlines=2000, gradbegin=colour5, gradend=colour6](2.3646572,-1.4787042)(2.3137617,-2.416411)(4.1486554,-2.2604098)(5.1705256,-1.6616218)(6.1923957,-1.0628339)(6.6671977,0.13458377)(5.2729506,0.22547626)(3.8787034,0.31636873)(2.4155526,-0.5409977)(2.3646572,-1.4787042)
\rput[bl](3.7,-1.1){$u \equiv 0$}
\rput[bl](5.2,-0.2){$\Omega^{\prime}$}
\rput[bl](1.3,-2.9){$F(X, \nabla u, D^2 u) = 0 \,\,\mbox{in} \,\, \{u>0\}$}
\rput[bl](6.1,1.0){$u = \varphi \,\, \mbox{on} \,\, \partial \Omega$}
\rput[bl](8.5,-3.5){$\mathbf{R}^n$}
\psline[linecolor=black, linewidth=0.04](5.483535,-1.4205831)(5.809861,-1.6918935)
\rput{-17.110401}(0.7442126,1.635026){\psdots[linecolor=black, dotstyle=triangle*, dotsize=0.12820485](5.8063936,-1.692062)}
\rput[bl](5.5,-2.0){$u_{\nu} = \psi \,\, \mbox{on} \,\, \partial\{u>0\}$}
\rput[bl](3.7,0.7){$\Omega$}
\end{pspicture}
}

\end{center}
\caption{Configuration for Free Boundary Problem}
\end{figure}

Hereafter in this paper, $F \colon \Omega \times \mathbb{R}^n \times \textrm{Sym}(n) \rightarrow \mathbb{R}$  is a fully nonlinear uniformly elliptic operator, i.e, there exist constants $\Lambda \geq \lambda > 0$ such that
\begin{equation}\label{UE} \tag{{\bf Unif. Ellip.}}
   \lambda \|N\| \le F(X,\overrightarrow{ p}, M+N) - F(X,\overrightarrow{ p}, M) \le \Lambda \|N\|,
\end{equation}
for all  $M,N \in \textrm{Sym}(n), N \ge 0, \overrightarrow{ p }\in \mathbb{R}^n  \,\ and  \,\ X \in \Omega$. As usual $\textrm{Sym}(n)$ denotes the set of all $n \times n$ symmetric matrices. Moreover, we must to observe the mapping $M \mapsto F(X, \overrightarrow{p} ,M)$ is monotone increasing in the natural order on $Sym(n)$ and Lipschitz continuous. Under such a structural condition, the theory of viscosity solutions provides a suitable notion for weak solutions.

\begin{definition}[{\bf Viscosity solution}] For an operator $F\colon \Omega \times \mathbb{R}^n \times \text{Sym}(n) \to \mathbb{R}$, we say a
function $u \in C^0(\Omega)$ is a viscosity supersolution (resp. subsolution) to
$$
     F(X, \nabla u, D^2 u) = f(X) \quad \mbox{in} \quad \Omega,
$$
if whenever we touch the graph of $u$ by below (resp. by above) at a point $Y \in \Omega$ by a smooth function $\phi$, there holds
$$
     F(Y, \nabla \phi(Y), D^2 \phi (Y)) \le  f(Y) \quad (\mbox{resp.} \ge f(Y)).
$$
Finally, we say $u$ is a viscosity solution if it is simultaneously a viscosity supersolution and subsolution.
\end{definition}

\begin{remark} All functions considered in the paper will be assumed continuous in $\overline{\Omega}$, namely $C$-viscosity solutions, see Caffarellli-Cabr\'{e} \cite{CC} and Teixeira \cite{Tei}. However, we also can to consider $L^p$-viscosity notion for such a solutions, see for example Winter \cite{Wint}.
\end{remark}

In \cite{RT}, several analytical and geometrical properties of such a fully nonlinear singular problem \eqref{eqRT} were established. Notwithstanding, regularity up to the boundary for approximating solutions has not been proven in the literature yet. This is the key goal of the present article. More precisely, we shall prove a uniform gradient estimate up to the boundary  for viscosity solutions of the singular perturbation problem
\begin{equation}\label{Equation Pe} \tag{\bf $\pe$}
\left\{
\begin{array}{rclcl}
F(X, \nabla u^{\varepsilon}, D^2 u^{\epsilon}) &=&  \zeta_{\epsilon}(u^{\epsilon}) & \mbox{in} & \Omega\\
u^{\epsilon} &=&\varphi & \mbox{on} & \partial \Omega,
\end{array}
\right.
\end{equation}
where we have: the singular reaction term $\zeta_{\varepsilon}(s) = \frac{1}{\varepsilon} \zeta \left(\frac{s}{\varepsilon}\right)$ for some non-negative  $\zeta \in C^{\infty}_{0}([0,1])$, a parameter $\varepsilon >0$, a non-negative $\varphi \in C^{1, \gamma}(\overline{\Omega})$, with $0<\gamma < 1$, and, a bounded $C^{1,1}$ domain $\Omega$ (or $\partial \Omega$ for short). Throughout this paper we will assume the following bounds: $\|\varphi\|_{C^{1,\gamma}(\overline{\Omega})} \le \mathcal{A}$ and $\|\zeta\|_{L^{\infty}([0, 1])} \le \mathcal{B}$.

\begin{theorem}[{\bf Global uniform Lipschitz estimate}] \label{principal2}
Let  $u^{\epsilon}$ be a viscosity solution to the singular perturbation problem \eqref{Equation Pe}. Then under the assumptions ${\bf (F1)-(F3)}$ there exists a constant $C(n, \lambda, \Lambda, b, \mathcal{A}, \mathcal{B},\Omega)>0$ independent of $\epsilon$, such that
$$
	\|\nabla u^{\epsilon}\|_{L^{\infty}(\overline{\Omega})} \le C.
$$
\end{theorem}


\medskip

Our new estimate  allows us to obtain existence for corresponding free boundary problem \eqref{Free}, keeping the prescribed boundary value data, see Theorem \ref{limFB}. Finally, we should emphasize our estimate generalizes the local gradient bound proven in \cite{Tei}, see also \cite{RT} for a rather complete local analysis of such a free boundary problem.

Although we have chosen to carry out the global analysis for the homogeneous case, the results presented in this paper can be adapted, under some natural adjustments, for the non-homogeneous case,
$$
\left\{
\begin{array}{rclcl}
 F(X, \nabla u^{\varepsilon}, D^2 u^{\epsilon})  &=& \zeta_{\epsilon}(u^{\epsilon}) + f_\varepsilon (X) & \mbox{in} & \Omega\\
u^{\epsilon}(X) &=&\varphi(X) & \mbox{on} & \partial \Omega,
\end{array}
\right.
$$
with $0\leq c_0 \le f_\varepsilon \le c_1$.

Our approach follows the pioneering work of Gurevich \cite{Gu}, where it is  introduced a new strategy to investigate uniform estimate up to boundary of two-phase singular perturbation problems involving linear elliptic operators of type $\mathcal{L}u = \partial_i(a_{ij}\partial_j u).$ This method has been successfully applied by Karakhanyan in \cite{K}  for the one-phase problem in the case involving nonlinear singular/degenerate elliptic operators of the $p$-Laplacian type $\Delta_{p}u^{\epsilon} = \zeta_{\epsilon}(u^{\epsilon}).$

\subsection{Organization of the article}

\hspace{0.6cm}The paper is organized of following way: In Section \ref{NoSta} we shall introduce the notation which will be used throughout of the paper, as well as we set up the structural assumptions for fully nonlinear elliptic operators. In Section \ref{ExisUniq} we discuss about the existence of appropriated notion of weak solutions to problem \eqref{Equation Pe}, namely \textit{Perron's type solutions}, see Theorem \ref{ExistMinSol}. The Section \ref{reglocal} is devoted to prove the main Theorem \ref{principal2}, for this reason it contains several keys Lemmas which are standard in the global regularity theory for elliptic operators in accordance with Gurevich \cite{Gu} and Karakhanyan \cite{K}, as well as Teixeira \cite{Tei} and Ricarte-Teixeira \cite{RT} for the corresponding local fully nonlinear singular perturbation theory. The free boundary problem, namely Theorem \ref{limFB} is obtained as consequence of global Lipschitz regularity. Finally, the last Section \ref{Append} is an Appendix where we prove two technical Lemmas (respectively Lemmas \ref{lemma2.0} and \ref{lemma2.1}) that play an important role in order to prove the main Theorem \ref{principal2} in Section \ref{reglocal}.

\section{Notation and statements}\label{NoSta}

\hspace{0.6cm}We shall introduce some notations and structural assumptions which we will use throughout this paper.
\begin{itemize}
\item[\checkmark] $n$ indicates the dimension of the Euclidean space.
\item[\checkmark] $\mathcal{H}_{+}$ is the half-space $\{X_n >0\}$.
\item[\checkmark] $\mathcal{H} \defeq \{X=(X_1, \ldots, X_n) \in \mathbb{R}^n : X_n=0\}$ indicates the hyperplane.
\item[\checkmark] $\hat{X}$ is the vertical projection of $X$ on $\mathcal{H}$.
\item[\checkmark] $\mathcal{C}_X \defeq \left\{Y \in H_{+} : |Y-\hat{Y}| \ge \frac{1}{2}|Y-X|\right\}$ is the cone with vertex at point $X \in \mathcal{H}$.
\item[\checkmark] $B_r(X)$ is the ball with center at $X$ and radius $r$, and, $B_r$ the ball $B_r(0)$.
\item[\checkmark] $B^{+}_{r}(X) \defeq B_{r}(X) \cap \mathcal{H}_{+}$.
\item[\checkmark] $B^{\prime}_{r}(X) $ is the ball with center at $X$ and radius $r$ in $\mathcal{H}$.
\end{itemize}

\begin{remark} Throughout this article \textit{Universal constants} are the ones depending only on the dimension, ellipticity and structural properties of $F$, i. e., $n, \lambda, \Lambda$ and $b$.
\end{remark}

Also, following classical notation, for constants $\Lambda \geq \lambda >0$ we denote by
$$
    \mathcal{P}^{+}_{\lambda,\Lambda}(M) \defeq \lambda  \cdot \sum_{e_i <0} e_i  + \Lambda \cdot \sum_{e_i >0} e_i \quad \mbox{and} \quad \mathcal{P}^{-}_{\lambda,\Lambda}(M) \defeq \lambda \cdot \sum_{e_i >0} e_i  + \Lambda \cdot \sum_{e_i <0} e_i
$$
the \textit{Pucci's extremal operators}, where $e_i = e_i(M)$ are the eigenvalues of $M \in Sym(n)$.

We shall introduce structural conditions that will be frequently used throughout of this paper:
\begin{enumerate}
\item[{\bf (F1)}](\textit{Ellipticity and Lipschitz regularity condition }) For all $M,N \in \textrm{Sym}(n)$, $\overrightarrow{p},\overrightarrow{q} \in \mathbb{R}^n$, $X \in \Omega$\label{F1}
$$
    \mathcal{P}^{-}_{\lambda, \Lambda}(M-N)-b|\overrightarrow{p}-\overrightarrow{q}| \le	F(X,\overrightarrow{p},M) - F(X,\overrightarrow{q},N) \le \mathcal{P}^{+}_{\lambda, \Lambda}(M-N) + b|\overrightarrow{p}-\overrightarrow{q}|
$$
\item[{\bf (F2)}] (\textit{Normalization condition}) We shall suppose that,
$$
    F(X, 0, 0) = 0
$$
\item[{\bf(F3)}] (\textit{Small oscillation condition}) We must to assume
$$
     \displaystyle \sup_{X_0 \in \Omega} \Theta_F(X, X_0)\ll 1
$$
where
$$
    \displaystyle \Theta_F(X, X_0) \defeq \sup_{M \in Sym(n)\backslash \{0\}} \frac{|F(X, 0, M)- F(X_0, 0, M)|}{\|M\|}
$$

\end{enumerate}

\begin{remark} Assumption ${\bf (F1)}$ is equivalent to notion of uniform ellipticity \ref{UE} when $\overrightarrow{p} = \overrightarrow{q}$. The assumption ${\bf (F2)}$ is not restrictive, since we can always redefine the operator in order to check it. The smallest regime on oscillation of $F$, namely condition ${\bf (F3)}$, depends only on universal parameters, see \cite{Wint}.
\end{remark}

\begin{example}[{\bf Isaacs type operators}]An example which we must have in mind are the Isaacs' operators from stochastic game theory
\begin{equation}\label{IO}
    \displaystyle	F(X, \overrightarrow{p}, M) \defeq \sup_{\alpha \in \mathfrak{A}} \inf_{\beta \in \mathfrak{B}}\left(\textrm{Tr}\left[A^{\alpha, \beta}(X)\cdot M\right] + \left\langle B^{\alpha, \beta}(X), \overrightarrow{p}\right\rangle \right) \quad \left(\mbox{resp.} \,\,\, \inf\limits_{\mathfrak{A}} \,\,\sup_{\mathfrak{B}} (\cdots)\right),
\end{equation}
where $A^{\alpha, \beta}$ is a family of measurable $n \times n$ real symmetric matrices with small oscillation satisfying
$$
	 \lambda\|\xi\|^2 \leq  \xi^TA^{\alpha, \beta}(X)\xi  \leq \Lambda \|\xi\|^2, \,\, \forall \,\, \xi \in \R^n \quad \mbox{and} \quad \|B^{\alpha, \beta}\|_{L^{\infty}(\Omega)}\leq b.
$$
\end{example}

\section{Existence of solutions}\label{ExisUniq}

\hspace{0.6cm}In this Section we shall comment on the existence of appropriated viscosity solutions to the singularly perturbed problem \eqref{Equation Pe}. Such a solutions are labeled by \textit{Perron's type solutions}.

\begin{theorem}[{\bf Perron's type method}, \cite{RT}]\label{PerMeth} Let $f \in C^{0, 1}([0, \infty)) $ be a bounded function. Suppose that there exist a viscosity sub-solution $\underline{u} \in C(\overline{\Omega}) \cap C^{0, 1}(\Omega)$ (respectively super-solution $\overline{u} \in C(\overline{\Omega}) \cap C^{0, 1}(\Omega)$) to $F(X, \nabla w, D^2w) = f(w)$ satisfying
$\underline{u} = \overline{u} = g \in C(\partial \Omega)$. Define the set of functions
$$
\mathcal{S} \defeq \left\{ v \in C(\overline{\Omega}) \;\middle|\; \begin{array}{c}
v \text{ is a viscosity super-solution to } \\
F(X, \nabla w, D^2w) = f(w) \text{ such that } \underline{u} \le v \le \overline{u}
\end{array}
\right\}.
$$
Then,
\begin{equation}\label{2.1}
	u(X) \defeq \inf_{v \in \mathcal{S}} v(X), \,\, \mbox{for} \,\, x \in \Omega
\end{equation}
is a continuous viscosity solution to $F(X, \nabla w, D^2w) = f(w)$ in $\Omega$ with $u=g$ continuously on $\partial \Omega$.
\end{theorem}

Existence of Perron's type solution to \eqref{Equation Pe} will follow by choosing $\underline{u} \defeq \underline{u}^{\varepsilon}$ and $\overline{u} \defeq \overline{u}^{\varepsilon}$ as solutions to the boundary value problems:
$$
\begin{array}{ccc}
\left\{
\begin{array}{rcccl}
F(X, \nabla \underline{u}^{\varepsilon}, D^2 \underline{u}^{\varepsilon}) & = & \sup\limits_{[0, \infty)} \zeta_{\varepsilon}(u^{\varepsilon}(X)) & \mbox{in} & \Omega\\
\underline{u}^{\varepsilon}(X) & = & \varphi(X) & \mbox{on} & \partial \Omega
\end{array}
\right.

& \mbox{and} &

\left\{
\begin{array}{rcccl}
F(X, \nabla \overline{u}^{\varepsilon}, D^2 \overline{u}^{\varepsilon}) & = & 0 & \mbox{in} & \Omega\\
\overline{u}^{\varepsilon}(X) & = & \varphi(X) & \mbox{on} & \partial \Omega,
\end{array}
\right.
\end{array}
$$

We must note that for each $\varepsilon>0$ fixed, existence of such a $\underline{u}^{\varepsilon}$ and $\overline{u}^{\varepsilon}$ follows as consequence of standard methods of sub and super solutions.
Moreover, we have that $\underline{u} \in C(\overline{\Omega})\cap C^{0, 1}(\Omega)$ and $\overline{u} \in C(\overline{\Omega})\cap C^{0, 1}(\Omega)$ are viscosity subsolution and supersolution to \eqref{Equation Pe} respectively. Finally, as consequence of the Theorem \ref{PerMeth} we have the following existence Theorem:

\begin{theorem}[{\bf Existence of Perron's type solutions, \cite{RT}}]\label{ExistMinSol} Given $\Omega \subset \R^n$ be a bounded Lipschitz domain and $g \in C(\partial \Omega)$ be a nonnegative boundary datum. There exists for each $\varepsilon>0$ fixed, a nonnegative Perron's type viscosity solution $u^{\varepsilon} \in C(\overline{\Omega})$ to \eqref{Equation Pe}.
\end{theorem}


\section{Optimal Lipschitz regularity}\label{reglocal}

\hspace{0.6cm}In this section, we shall present the proof of Theorem \ref{principal2}. Thus let us assume the assumptions of problem \eqref{Equation Pe}.

We make a pause as to discuss some remarks which will be important  throughout this work. Firstly it is important to highlight that is always possible to perform a change of variables to flatten the boundary. Indeed, if $\partial \Omega$ is a $C^{1, 1}$ set, the part of $\Omega$ near $\partial \Omega$ can be covered with a finite collection of regions that can be mapped onto half-balls by diffeomorphisms (with portions of $\partial \Omega$ being mapped onto the ``flat" parts of the boundaries of the half-balls). Hence,  we can use a smooth mapping, reducing this way the general case to that one on $B^{+}_{1}$, and, the boundary data would be given on $B^{\prime}_1$.


Previously we start the proof of the global Lipschitz estimative, we need to assure the non-negativity and boundedness of solutions to \eqref{Equation Pe}. This statement is a consequence of the Alexandroff-Bekelman-Pucci Maximum Principle, see \cite{CC} for more details.

\begin{lemma}[{\bf Nonnegativity and boundedness, \cite{RT} and \cite{Tei}}] \label{ABP}
Let $u^{\varepsilon}$ be a viscosity solution to  \eqref{Equation Pe}. Then there exists a universal constant $C>0$ such that
$$
    0 \le u^{\varepsilon}(X) \le C\|\varphi\|_{L^{\infty}(\overline{\Omega})} \quad \mbox{in} \quad \Omega.
$$
\end{lemma}

\bigskip
 We will now establish a universal bound for the Lipschitz norm of $u^{\varepsilon}$ up to the boundary.  The proof will be divided in two cases.
\vspace{0.2cm}
\begin{center}
 {\large \bf Case 1: Lipschitz regularity up to the boundary in the region $\{0 \le u^{\varepsilon} \le \varepsilon\}$}.
\end{center}
\vspace{0.2cm}
\begin{theorem} \label{prop1}
Let $u^{\epsilon}$ be a viscosity solution to \eqref{Equation Pe}. For $X \in \{0 \le u^{\varepsilon} \le \varepsilon\} \cap B^{+}_{\frac{1}{2}}$ there exists a universal constant $C_1>0$ independent of $\varepsilon$ such that
$$
	|\nabla u^{\epsilon}(X)| \le C_1.
$$
\end{theorem}
\begin{proof}
 We denote by
$$
	\delta(X) \defeq \textrm{dist}(X, \mathcal{H})
$$
the vertical distance. If $\delta(X) \ge \epsilon$, then $B_{\varepsilon}(X) \subset B^{+}_{1}$ for $\varepsilon \ll 1$. Therefore, from local gradient bounds \cite{RT, Tei} ,  there exists a universal constant $C_0>0$ independent of $\varepsilon$, such that
$$
	|\nabla u^{\epsilon}(X)| \le C_0.
$$
On the other hand, if $\delta(X) < \epsilon$, then it is sufficient to prove that there exists a
universal constant $C_0>0$ independent of $\varepsilon$, such that
\begin{equation} \label{new}
	u^{\epsilon}(\hat{X}) \le C_0 \epsilon.
\end{equation}
Indeed, suppose that \eqref{new} holds. Consider
$h \colon \overline{B}^{+}_{1} \rightarrow \mathbb{R}$ to  be the viscosity solution to the Dirichlet problem
$$
\left\{
\begin{array}{rclcl}
F(Y, \nabla h, D^2 h)  &=&  0 & \mbox{in} & B^{+}_{1}\\
h & = & u^{\epsilon}  & \mbox{on} &  \partial B^{+}_{1}.
\end{array}
\right.
$$
From $C^{1,\alpha}$ regularity estimates up to the boundary (see for instance \cite[Theorem 3.1]{Wint}), we know that $h \in C^{1,\alpha}\left(\overline{B}^{+}_{\frac{3}{4}}\right)$
with the following estimate
$$
		| \nabla h(Y)| \le c \left(\|h\|_{L^{\infty}(B^{+}_{1})} + \|\varphi\|_{C^{1,\gamma}(B^{\prime}_{1})}\right) \le C \quad \textrm{in} \quad B^{+}_{\frac{3}{4}}
$$
and by Comparison Principle we  have
$$
   u^{\epsilon} \le h \quad \mbox{in} \quad B^{+}_{1}.
$$
Hence, it follows from assumption \eqref{new} that
$$
	u^{\epsilon}(Y) \le h(Y) \le h(\hat{X}) + C |Y - \hat{X}| \le C \epsilon \quad \textrm{if} \quad Y \in B^{+}_{2 \epsilon}(\hat{X})
$$
Then, again applying $C^{1,\alpha}$ regularity estimates from \cite{Wint}, we obtain
$$
	|\nabla u^\varepsilon(X)| \le C_0(n, \lambda, \Lambda, b, \mathcal{B}).
$$

In order to prove \eqref{new} suppose, by purpose of contradiction, there exists $\epsilon >0$ such that
$$
	u^{\epsilon}(\hat{X}) \ge k \epsilon  \quad \text{ for } \quad k \gg 1.
$$
We shall denote
$$
  r_0 \defeq \textrm{dist}(\hat{X}, \{0 \le u^{\varepsilon} \le \varepsilon\}).
$$
Consider $X_0 \in \{0 \le u^{\varepsilon} \le \varepsilon\} \cap \partial B_{r_0}^{+}(\hat{X})$ a point
to which the distance is achieved, i.e.,
$$
	r_0 = |X_0 - \hat{X}|.
$$
Thereafter, let
$
	\mathcal{C}_{\hat{X}}
$
be the cone with vertex at $\hat{X} \in \mathcal{H}$. Suppose initially that $X_0 \in \mathcal{C}_{\hat{X}}$ then $B_{\frac{r_0}{2}}(X_0) \subset B^{+}_{1}$ .

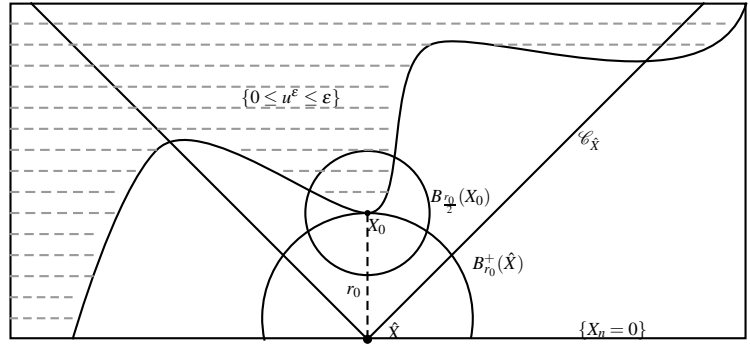
\begin{figure}[h]
\begin{center}

%
%
\psscalebox{0.7 0.7} 
{
\begin{pspicture}(0,-4.244144)(14.419121,4.244144)
\definecolor{colour0}{rgb}{0.6,0.6,0.6}
\psframe[linecolor=black, linewidth=0.04, dimen=outer](14.4,2.244144)(0.4,-4.155856)
\rput{-91.288635}(11.116827,3.3578575){\psarc[linecolor=black, linewidth=0.04, dimen=outer](7.2,-3.755856){2.0}{80.20546}{283.56296}}
\psbezier[linecolor=black, linewidth=0.04](1.6,-4.155856)(1.6,-4.155856)(2.4,-1.3558561)(3.2,-0.555856)(4.0,0.24414398)(6.4,-1.755856)(7.2,-1.755856)(8.0,-1.755856)(7.490642,1.0281296)(8.4,1.444144)(9.309358,1.8601584)(13.720721,0.031729687)(14.4,2.244144)
\psline[linecolor=black, linewidth=0.04, linestyle=dashed, dash=0.17638889cm 0.10583334cm](7.2,-1.755856)(7.2,-4.155856)
\pscircle[linecolor=black, linewidth=0.04, dimen=outer](7.2,-1.755856){1.2}
\rput[bl](7.2,-2.1558561){$X_0$}
\psline[linecolor=black, linewidth=0.04](7.2,-4.155856)(0.8,2.244144)
\psline[linecolor=black, linewidth=0.04](7.2,-4.155856)(13.6,2.244144)
\rput{4.909062}(-0.32775682,-0.6656109){\rput[bl](7.6,-4.155856){$\hat{X}$}}
\rput{-219.83366}(10.066853,-11.959208){\psdots[linecolor=black, dotsize=0.15972319](7.2,-4.155856)}
\psdots[linecolor=black, dotsize=0.1686082](7.2,-4.155856)
\rput[bl](6.8,-3.355856){$r_0$}
\rput[bl](8.4,-1.755856){$B_{\frac{r_0}{2}}(X_0)$}
\rput[bl](9.2,-2.955856){$B_{r_0}^{+}(\hat{X})$}
\rput[bl](11.2,-0.555856){$\mathcal{C}_{\hat{X}}$}
\rput[bl](11.2,-4.155856){$\{X_n=0\}$}
\rput{134.29721}(10.971626,-8.135344){\psdots[linecolor=black, dotsize=0.11533203](7.2,-1.755856)}
\rput[bl](4.8,0.24414398){$\{0\leq u^{\varepsilon}\leq \varepsilon\}$}
\pscustom[linecolor=black, linewidth=0.04]
{
\newpath
\moveto(6.0,-0.15585601)
}

\pscustom[linecolor=black, linewidth=0.04]
{
\newpath
\moveto(8.4,-0.15585601)
}

\pscustom[linecolor=black, linewidth=0.04]
{
\newpath
\moveto(4.4,-0.555856)
}
\psline[linecolor=colour0, linewidth=0.04, linestyle=dashed, dash=0.17638889cm 0.10583334cm](0.4,-3.755856)(1.6,-3.755856)
\psline[linecolor=colour0, linewidth=0.04, linestyle=dashed, dash=0.17638889cm 0.10583334cm](0.4,-3.355856)(1.6,-3.355856)
\psline[linecolor=colour0, linewidth=0.04, linestyle=dashed, dash=0.17638889cm 0.10583334cm](0.4,-2.955856)(2.0,-2.955856)
\psline[linecolor=colour0, linewidth=0.04, linestyle=dashed, dash=0.17638889cm 0.10583334cm](0.4,-2.555856)(2.0,-2.555856)
\psline[linecolor=colour0, linewidth=0.04, linestyle=dashed, dash=0.17638889cm 0.10583334cm](0.4,-2.1558561)(2.0,-2.1558561)
\psline[linecolor=colour0, linewidth=0.04, linestyle=dashed, dash=0.17638889cm 0.10583334cm](0.4,-1.755856)(2.4,-1.755856)
\psline[linecolor=colour0, linewidth=0.04, linestyle=dashed, dash=0.17638889cm 0.10583334cm](0.4,-1.3558561)(2.4,-1.3558561)
\psline[linecolor=colour0, linewidth=0.04, linestyle=dashed, dash=0.17638889cm 0.10583334cm](0.4,-0.955856)(2.8,-0.955856)
\psline[linecolor=colour0, linewidth=0.04, linestyle=dashed, dash=0.17638889cm 0.10583334cm](0.4,-0.555856)(3.2,-0.555856)
\psline[linecolor=colour0, linewidth=0.04, linestyle=dashed, dash=0.17638889cm 0.10583334cm](0.4,-0.15585601)(7.6,-0.15585601)
\psline[linecolor=colour0, linewidth=0.04, linestyle=dashed, dash=0.17638889cm 0.10583334cm](0.4,0.24414398)(7.6,0.24414398)
\psline[linecolor=colour0, linewidth=0.04, linestyle=dashed, dash=0.17638889cm 0.10583334cm](0.4,0.644144)(7.6,0.644144)
\psline[linecolor=colour0, linewidth=0.04, linestyle=dashed, dash=0.17638889cm 0.10583334cm](0.4,1.044144)(8.0,1.044144)
\psline[linecolor=colour0, linewidth=0.04, linestyle=dashed, dash=0.17638889cm 0.10583334cm](0.4,1.444144)(8.4,1.444144)
\psline[linecolor=colour0, linewidth=0.04, linestyle=dashed, dash=0.17638889cm 0.10583334cm](0.4,1.844144)(14.0,1.844144)
\psline[linecolor=colour0, linewidth=0.04, linestyle=dashed, dash=0.17638889cm 0.10583334cm](9.6,1.444144)(13.6,1.444144)
\psline[linecolor=colour0, linewidth=0.04, linestyle=dashed, dash=0.17638889cm 0.10583334cm](7.6,-0.555856)(4.8,-0.555856)
\psline[linecolor=colour0, linewidth=0.04, linestyle=dashed, dash=0.17638889cm 0.10583334cm](7.6,-0.955856)(5.6,-0.955856)
\psline[linecolor=colour0, linewidth=0.04, linestyle=dashed, dash=0.17638889cm 0.10583334cm](7.6,-1.3558561)(6.4,-1.3558561)
\end{pspicture}
}

\end{center}
\caption{Geometric argument for the case $X_0 \in \mathcal{C}_{\hat{X}}$.}
\end{figure}
Now,
 let us define, $v^{\epsilon} : B_1 \rightarrow \mathbb{R}$ by
$$
	v^{\epsilon}(Y) \defeq \frac{u^{\epsilon}(X_0 + (r_0 /2)Y)}{\epsilon}.
$$
Therefore, $v^{\epsilon}$ satisfies in the viscosity sense
$$
	 F_{\varepsilon}(Y, \nabla v^{\varepsilon}, D^2 v^{\epsilon}) =  \frac{1}{\epsilon^{2}} \left(\frac{r_0}{2}\right)^{2} \zeta(v^{\epsilon}) \defeq \mathfrak{g}(Y),
$$
where
\begin{equation} \label{escale}
F_{\varepsilon}(Y, \overrightarrow{p}, M) \defeq \frac{1}{\varepsilon} \left(\frac{r_0}{2}\right)^2 F \left(X_0 + \frac{r_0}{2} Y,   \frac{2\varepsilon}{r_0} \cdot p, \varepsilon \left(\frac{2}{r_0}\right)^2 M\right).
\end{equation}




Now note that $\mathfrak{g} \in L^{\infty}(B_1)$, since $r_0 < \epsilon$ and $F_{\varepsilon}$ satisfies ${\bf (F1)-(F3)}$ with constant $\tilde{b} = \frac{r_0}{2}\cdot b$.
Moreover, since $v^{\epsilon}(0) \le 1$ it follows from Harnack inequality that
$$
    v^{\epsilon}(Y) \le c \quad \mbox{for} \quad Y \in B_{\frac{1}{2}},
$$
i.e.,
$$
	u^{\epsilon}(X) \le c \epsilon, \quad X \in B_{\frac{r_0}{4}}(X_0).
$$
Consider now $Z \in B^{\prime   }_{r_0}(\hat{X})$. It follows that
$$
	\varphi(Z) \ge \varphi(\hat{X}) - \mathcal{A} \cdot |Z-\hat{X}| \ge k \epsilon -r_0 \cdot  \mathcal{A}  \ge (k-\mathcal{A}) \epsilon,
$$
since $r_0 < \epsilon$. Define the scaled function $w^{\epsilon} : B^{+}_{1} \rightarrow \mathbb{R}$,
$$
	w^{\epsilon}(Y) \defeq \frac{u^{\epsilon}(\hat{X} + r_0 Y)}{\epsilon}.
$$
It readily follows that

$$
\left\{
\begin{array}{rcl}
F_{\varepsilon}(Y, \nabla w^{\varepsilon}, D^2 w^{\epsilon})  =  0 & \mbox{in} & B^{+}_{1}\\
w^{\epsilon}(Y) \ge k-\mathcal{A} & \mbox{on} &  B^{\prime}_{1},
\end{array}
\right.
$$
where $F_{\varepsilon}$ is as in \eqref{escale}. Therefore according to Lemma  \ref{lemma2.0},
$$
    w^{\epsilon}(Y) \ge c(k-\mathcal{A}) \quad \mbox{in} \quad B^{+}_{\frac{3}{4}}.
$$
In other words, we have reached that
$$
	u^{\epsilon}(X) \ge c \epsilon (k-\mathcal{A}) \quad \textrm{in} \quad B^{+}_{\frac{3r_0}{4}}(\hat{X}).
$$
Hence
$$
	c \epsilon (k-\mathcal{A}) \le u^{\epsilon}(Z_0) \le c \epsilon,  \quad \forall \,\,\,\, Z_0 \in \partial B_{\frac{3r_0}{4}}^{+}(\hat{X}) \cap \partial B_{\frac{r_0}{4}}(X_0),
$$
which leads to a contradiction for $k \gg 1$.

On the one hand if $X_0 \not\in \mathcal{C}_{\hat{X}}$, choose $X_1 \in \{0 \le u^{\varepsilon} \le \varepsilon\}$
such that
$$
	r_1 \defeq \textrm{dist}(\hat{X}_{0}, \{0 \le u^{\varepsilon} \le \varepsilon\}) = |\hat{X}_{0} - X_1|.
$$
From triangular inequality and the fact that $r_1 \le \frac{r_0}{2}$ we have
$$
	|X_1 - \hat{X}| \le |X_1 - \hat{X}_{0}| + |\hat{X}_{0} - \hat{X}| \le r_1 +r_0 \le \frac{r_0}{2} + r_0.
$$
If $X_1 \in \mathcal{C}_{\hat{X}_{0}}$ the result follows from previous analysis. Otherwise,
let $X_{2}$ be such that
$$
	r_2 \defeq \textrm{dist}(\hat{X}_{1}, \{0 \le u^{\varepsilon} \le \varepsilon\}) = |\hat{X}_{1} - X_2|.
$$
As before we have
 $$
	|X_2 - \hat{X}| \le |\hat{X}_{1} - X_2| + |\hat{X}_{1} - \hat{X}| \le \frac{r_0}{4} + \frac{r_0}{2} + r_0,
$$
since $r_2 \le \frac{r_1}{2} \le \frac{r_0}{4}$. Observe that this process must finish up within a finite number of steps.  Indeed, suppose that we have a sequence of points
$
	X_j \in \partial \{0\le u^{\varepsilon} \le \varepsilon\}, \quad X_{j+1} \not \in \mathcal{C}_{\hat{X}_{j}} \,\,\, (j=1,2,\ldots)
$
satisfying,
$$
	r_{j+1} \defeq \textrm{dist}(\hat{X}_{j}, \{0 \le u^{\varepsilon} \le \varepsilon\}) = |X_{j+1} - \hat{X}_{j}|
$$
and
\begin{equation} \label{ind}
	r_{j+1} \le \frac{r_{j}}{2} \le \frac{r_0}{2^{j+1}}.
\end{equation}
Thus, it follows from \eqref{ind} that
$$
	|X_j - \hat{X}| \le r_0 + r_0 \sum_{i=1}^{j} \frac{1}{2^i} \le 2r_0.
$$
Therefore, up to a subsequence, $X_j \to X_{\infty} \in B^{\prime}_{2r_0}(\hat{X})$ with $\varphi(X_{\infty}) = \varepsilon$. However,
$$
	\varphi(X_{\infty}) \ge \varphi(\hat{X}) - \mathcal{A} \cdot | \hat{X}- X_{\infty} | \ge \varepsilon (k-2 \mathcal{A}) \gg \varepsilon
 $$
for $k \gg 1$ which drives us to a contradiction, and, hence the assertion \eqref{new} is proved.
\end{proof}

\begin{figure}[h]
\begin{center}
\psscalebox{0.5 0.5} 
{
\begin{pspicture}(0,-4.0261793)(15.637154,4.0261793)
\definecolor{colour0}{rgb}{0.4,0.4,0.4}
\psframe[linecolor=black, linewidth=0.04, dimen=outer](15.62,4.0261793)(0.02,-3.9738207)
\rput{-77.596115}(8.217295,3.0733209){\psarc[linecolor=black, linewidth=0.04, dimen=outer](6.02,-3.5738208){2.8}{69.98726}{265.34958}}
\psbezier[linecolor=black, linewidth=0.04](15.62,2.0261793)(14.42,4.0261793)(8.82,-5.5738206)(10.42,-1.9738208)(12.02,1.6261792)(8.42,-2.7738209)(8.42,-1.9738208)(8.42,-1.1738209)(9.22,1.6261792)(7.62,1.6261792)(6.02,1.6261792)(5.62,-0.7738208)(3.62,0.4261792)(1.62,1.6261792)(1.9644444,3.5690362)(0.02,2.8261793)
\psline[linecolor=colour0, linewidth=0.04, linestyle=dotted, dotsep=0.10583334cm](8.42,-1.9738208)(8.42,-3.9738207)
\rput[bl](6.02,-3.5738208){$\hat{X}$}
\rput[bl](8.82,-1.5738208){$X_0$}
\rput[bl](8.02,-3.5738208){$\hat{X_0}$}
\rput{-79.40087}(10.743901,5.393249){\psarc[linecolor=black, linewidth=0.04, dimen=outer](8.62,-3.7738209){1.8}{72.382065}{265.051}}
\rput[bl](10.42,-2.7738209){$X_1$}
\psline[linecolor=colour0, linewidth=0.04, linestyle=dotted, dotsep=0.10583334cm](8.42,-3.9738207)(10.02,-2.7738209)
\psline[linecolor=colour0, linewidth=0.04, linestyle=dotted, dotsep=0.10583334cm](6.02,-3.9738207)(8.42,-1.9738208)
\rput[bl](9.22,-3.1738207){$r_1$}
\rput[bl](7.22,-3.1738207){$r_0$}
\psline[linecolor=colour0, linewidth=0.04, linestyle=dotted, dotsep=0.10583334cm](10.02,-2.7738209)(10.02,-3.9738207)
\rput[bl](9.62,-3.9738207){$\hat{X_1}$}
\rput[bl](0.42,1.2261792){$\mathcal{C}_{\hat{X}}$}
\rput[bl](5.22,2.4261792){$\{0 \leq u^{\varepsilon} \leq \varepsilon\}$}
\psline[linecolor=black, linewidth=0.04, linestyle=dashed, dash=0.17638889cm 0.10583334cm](8.42,-3.9738207)(15.22,0.0261792)
\psline[linecolor=black, linewidth=0.04, linestyle=dashed, dash=0.17638889cm 0.10583334cm](8.42,-3.9738207)(1.62,0.0261792)
\psline[linecolor=black, linewidth=0.04, linestyle=dashed, dash=0.17638889cm 0.10583334cm](6.02,-3.9738207)(0.02,2.8261793)
\rput[bl](13.62,-1.5738208){$\mathcal{C}_{\hat{X_0}}$}
\rput[bl](12.02,-3.9738207){$\{X_n = 0\}$}
\psline[linecolor=black, linewidth=0.04, linestyle=dashed, dash=0.17638889cm 0.10583334cm](6.02,-3.9738207)(12.02,2.8261793)
\psline[linecolor=black, linewidth=0.04, linestyle=dotted, dotsep=0.10583334cm](0.02,3.6261792)(15.62,3.6261792)
\psline[linecolor=black, linewidth=0.04, linestyle=dotted, dotsep=0.10583334cm](0.02,3.2261791)(15.62,3.2261791)
\psline[linecolor=black, linewidth=0.04, linestyle=dotted, dotsep=0.10583334cm](1.22,2.8261793)(15.62,2.8261793)
\psline[linecolor=black, linewidth=0.04, linestyle=dotted, dotsep=0.10583334cm](1.62,2.4261792)(15.62,2.4261792)
\psline[linecolor=black, linewidth=0.04, linestyle=dotted, dotsep=0.10583334cm](2.02,2.0261793)(14.42,2.0261793)
\psline[linecolor=black, linewidth=0.04, linestyle=dotted, dotsep=0.10583334cm](2.42,1.6261792)(14.02,1.6261792)
\psline[linecolor=black, linewidth=0.04, linestyle=dotted, dotsep=0.10583334cm](2.82,1.2261792)(6.42,1.2261792)
\psline[linecolor=black, linewidth=0.04, linestyle=dotted, dotsep=0.10583334cm](8.42,1.2261792)(13.62,1.2261792)
\psline[linecolor=black, linewidth=0.04, linestyle=dotted, dotsep=0.10583334cm](3.22,0.8261792)(6.02,0.8261792)
\psline[linecolor=black, linewidth=0.04, linestyle=dotted, dotsep=0.10583334cm](4.02,0.4261792)(5.62,0.4261792)
\psline[linecolor=black, linewidth=0.04, linestyle=dotted, dotsep=0.10583334cm](8.82,0.8261792)(13.22,0.8261792)
\psline[linecolor=black, linewidth=0.04, linestyle=dotted, dotsep=0.10583334cm](8.82,0.4261792)(12.82,0.4261792)
\psline[linecolor=black, linewidth=0.04, linestyle=dotted, dotsep=0.10583334cm](8.82,0.0261792)(12.42,0.0261792)
\psline[linecolor=black, linewidth=0.04, linestyle=dotted, dotsep=0.10583334cm](8.82,-0.3738208)(12.02,-0.3738208)
\psline[linecolor=black, linewidth=0.04, linestyle=dotted, dotsep=0.10583334cm](8.82,-0.7738208)(10.02,-0.7738208)
\psline[linecolor=black, linewidth=0.04, linestyle=dotted, dotsep=0.10583334cm](10.82,-0.7738208)(11.62,-0.7738208)
\psline[linecolor=black, linewidth=0.04, linestyle=dotted, dotsep=0.10583334cm](8.42,-1.1738209)(9.62,-1.1738209)
\psline[linecolor=black, linewidth=0.04, linestyle=dotted, dotsep=0.10583334cm](10.82,-1.1738209)(11.62,-1.1738209)
\psline[linecolor=black, linewidth=0.04, linestyle=dotted, dotsep=0.10583334cm](8.42,-1.5738208)(9.22,-1.5738208)
\psline[linecolor=black, linewidth=0.04, linestyle=dotted, dotsep=0.10583334cm](10.82,-1.5738208)(11.22,-1.5738208)
\psline[linecolor=black, linewidth=0.04, linestyle=dotted, dotsep=0.10583334cm](10.42,-1.9738208)(10.82,-1.9738208)
\psdots[linecolor=black, dotsize=0.0125](6.02,-3.9738207)
\psdots[linecolor=black, dotsize=0.10625](6.02,-3.9738207)
\psdots[linecolor=black, dotsize=0.10625](8.42,-3.9738207)
\psdots[linecolor=black, dotsize=0.099975586](10.02,-3.9738207)
\end{pspicture}
}

\end{center}
\caption{Geometric argument for the inductive process.}
\end{figure}

\vspace{0.25cm}
\begin{center}
{\large \bf Case 2: \,\,  Lipschitz regularity in the region $B^{+}_{1/8} \setminus \{0\le u^{\varepsilon} \le \varepsilon\}$}.
\end{center}
\vspace{0.2cm}

\begin{theorem}
 Let $u^{\varepsilon}$ be a viscosity solution to  \eqref{Equation Pe}. If  $X \in B^{+}_{\frac{1}{8}} \cap \{u^{\varepsilon} > \varepsilon\}$, then there exists a constant $C_0 = C_0(n, \lambda, \Lambda, b, \mathcal{A})>0$ such that
 $$
 	|\nabla u^{\varepsilon}(X)| \le C_0.
 $$
\end{theorem}

The proof of the theorem consists in analysing three possible cases (Lemmas \ref{prop2}, \ref{prop2.1}, \ref{prop2.2} below). Henceforth we shall use the following notation
$$
	 \delta_{\varepsilon}(X) \defeq \textrm{dist}(X, \{0 \le u^{\varepsilon} \le \varepsilon\}) \quad \textrm{and} \quad\delta(X) \defeq \textrm{dist}(X, \mathcal{H}).
$$

The next result is decisive in our approach.

\begin{lemma}\label{l5}
Let $u^{\epsilon}$ be a viscosity solution to \eqref{Equation Pe} with $\varphi \in C^{1,\gamma}(B^{\prime}_{1})$. Then, for all $X \in B^{\prime}_{\frac{1}{4}} \cap \{u^{\varepsilon}> \varepsilon\}$,  there exists a constant $c_0=c_0(n, \lambda, \Lambda, b) >0$ such that
$$
	\varphi(X)\le \epsilon + c_0 \cdot  \delta_{\varepsilon}(X).
$$
\end{lemma}
\begin{proof}
 Let us suppose for sake of contradiction that there exists an $\epsilon>0$ and
 $X_0 \in B^{\prime}_{\frac{1}{4}} \setminus \{0\le u^{\varepsilon} \le \varepsilon\}$ such that
$$
  \varphi(X_0) \ge \epsilon + k \cdot \delta_{\varepsilon}(X_0)
$$
holds for $k \gg 1$, large enough. Let
$Z=Z_{\epsilon} \in \partial \{0 \le u^{\varepsilon} \le \varepsilon\}$ be a point to which the distance is achieved, i.e.
$$
    \delta_{\varepsilon} \defeq \delta_{\epsilon}(X_0) = |X_0 - Z|.
$$
We have two cases to analyse: If $Z \in \mathcal{C}_{X_0}$, then the normalized function
$v^{\epsilon} \colon B^{+}_{1} \rightarrow \mathbb{R}$ given by
$$
	v^{\epsilon}(Y) \defeq \frac{u^{\epsilon}(X_0 + \delta_{\epsilon} Y)-\varepsilon}{\delta_{\epsilon}}
$$
satisfies
$$
    F_{\varepsilon}(Y, \nabla v^{\varepsilon}, D^{2} v^{\epsilon}) = 0 \quad \mbox{in} \quad B^{+}_{1}
$$
 in the viscosity sense, where
$$
F_{\varepsilon}(Y, \overrightarrow{p}, M) \defeq \delta_{\varepsilon} F\left(X_0+\delta_{\varepsilon}Y, \overrightarrow{p}, \frac{1}{\delta_{\varepsilon}} M\right).
$$
As in Theorem \ref{prop1}, $F_{\varepsilon}$ satisfies ${\bf (F1)-(F3)}$ with constant $\tilde{b} = \delta_{\varepsilon} b$. Moreover,
$$
   v^{\epsilon}(Y) \ge 0 \quad \mbox{in} \quad B^{+}_{1}.
$$
Now, for any $X \in B^{\prime}_{\delta_{\epsilon}}(X_0)$
we should have for $k \gg 1$,
\begin{eqnarray*}
	\varphi(X) &\ge&  \varphi(X_0) - \mathcal{A} \delta_{\varepsilon} \ge  \varepsilon + k \delta_{\varepsilon} - \mathcal{A} \delta_{\varepsilon} \\
	&\ge& \varepsilon + \frac{k}{2} \delta_{\varepsilon},
\end{eqnarray*}
i.e,
$$
	\frac{\varphi(X_0 + \delta_{\varepsilon} Y) - \varepsilon}{\delta_{\varepsilon}} \ge \frac{k}{2} \quad \textrm{in} \,\,\, B^{\prime}_1.
$$
In other words,
$$
   v^{\epsilon}(Y) \ge c k \quad  \forall \,\, Y \in B'_{1}.
$$
Hence, from Lemma \ref{lemma2.0} we have that
$$
v^{\epsilon} \ge ck \quad \mbox{in} \quad B^{+}_{\frac{3}{4}}.
$$
In a more precise manner,
\begin{equation} \label{m1}
	u^{\epsilon}(X) \ge \epsilon + C k \delta_{\epsilon}, \quad X \in B^{+}_{\frac{3\delta_{\epsilon}}{4}}(X_0).
\end{equation}

From now on, let us consider $\tilde{B} \defeq B_{\frac{\delta_{\epsilon}}{4}}(\mathbf{P})$, where
$\mathbf{P} = \mathbf{P}_{\epsilon}\defeq  Z + \frac{X_0 - Z}{4}$. If we define
$\omega^{\varepsilon} \defeq u^{\epsilon}-\epsilon$, then since $Z \in \partial \tilde{B}$, it follows that
\begin{eqnarray}	
	  F_{\epsilon}(X, \nabla \omega^{\varepsilon} , D^2 \omega^{\epsilon})=0 \quad \text{ in } \quad \tilde{B}, \label{LL1}\\
	  \omega^{\epsilon}(Z)=u^{\varepsilon}(Z)-\varepsilon =0, \label{LL2}\\
	  \frac{\partial \omega^{\varepsilon}}{\partial \nu}(Z) \le |\nabla \omega^{\varepsilon}(Z)| \le C. \label{LL3}
\end{eqnarray}	
 Therefore, from \eqref{LL1}-\eqref{LL3} we can apply Lemma \ref{lemma2.1}, which gives
 $$
    \omega^{\varepsilon}(\mathbf{P}) \le C_0 \cdot \delta_{\varepsilon},
 $$
i.e.,
\begin{equation}\label{m2}
	u^{\epsilon}(\mathbf{P}) \le \epsilon + C \delta_{\epsilon}.
\end{equation}
At  a point $\mathbf{P}$ on $\partial B_{\frac{3 \delta_{\epsilon}}{4}}^{+}(X_0)$
we have (according to \eqref{m1} and \eqref{m2})
$$
	\epsilon + k  c  \delta_{\epsilon} \le u^{\epsilon}(\mathbf{P}) \le \varepsilon + C_0 \delta_{\epsilon}
$$
which gives a contradiction if $k$ has been chosen large enough.

The second case, namely $Z \not\in \mathcal{C}_{X_{0}}$, it is treated similarly as in Theorem \ref{prop1} and for this reason we omit the details here.
\end{proof}

\begin{lemma}\label{prop2}
Let $u^{\varepsilon}$ be a viscosity solution to \eqref{Equation Pe} and  $X \in B^{+}_{\frac{1}{8}} \cap \{u^{\varepsilon} > \varepsilon\}$ such that $\delta_{\varepsilon}(X) \le \delta(X)$. Then there exists a universal constant $C_0>0$, such that
$$
	|\nabla u^{\epsilon}(X)| \le C_0.
$$
\end{lemma}
\begin{proof}
We may assume with no loss of generality that $\delta_{\epsilon}(X) \le \frac{1}{8}$. Otherwise, if we suppose
that $\delta_{\epsilon}(X) > \frac{1}{8}$, then the result would follow from \cite{RT, Tei}.  From now on,
we select $X_{\epsilon} \in \partial \{0 \le u^{\varepsilon} \le \varepsilon\}$ a point which achieves the distance, i.e.,
$$
	\delta_{\epsilon} \defeq \delta_{\varepsilon}(X) =  |X-X_{\epsilon}|.
$$
Since
$$
	|X_{\epsilon}| \le |X| + \delta_{\epsilon} \le \frac{1}{4},
$$
we must have that  $X_{\varepsilon} \in B^{+}_{\frac{1}{4}} \cap \{0 \le u^{\varepsilon} \le \varepsilon\}$.
This way, by applying Theorem \ref{prop1}, there exists a constant $C_1=C(n,\lambda,\Lambda, b, \mathcal{A}, \mathcal{B})>0$ such that
$$
	|\nabla u^{\epsilon}(X_{\epsilon})| \le C_1.
$$
By defining the re-normalized  function $v^{\epsilon} : B_1 \rightarrow \mathbb{R}$ as
$$
	v^{\epsilon}(Y) \defeq \frac{u^{\epsilon}(X + \delta_{\epsilon} Y) - \epsilon}{\delta_{\epsilon}}.
$$
Then, as before $v^{\epsilon}$ satisfies
\begin{eqnarray}	
	  F_{\epsilon}(Y, \nabla v^{\varepsilon} , D^2 v^{\epsilon})=0 \quad \text{ in } \quad B_1, \label{L1}\\
	  v^{\epsilon}(Y_{\epsilon})=0, \label{L2}\\
	  |\nabla v^{\epsilon}(Y_{\epsilon})| \le C_1, \label{L3}\\
	  v^{\epsilon}(Y) \ge 0 \,\,\, \text{for}\,\,\,  Y \in B_1 \label{L4},
\end{eqnarray}	
where
$$
   F_{\epsilon}(Y,\overrightarrow{ p} , M) \defeq \delta_{\epsilon}F\left(X+ \delta_{\varepsilon} Y, \overrightarrow{p} , \frac{1}{\delta_{\epsilon}}M\right)
\quad \mbox{and} \quad  Y_{\epsilon} := \frac{X_{\epsilon} -X}{\delta_{\epsilon}} \in \partial B_1.
$$
From \eqref{L1}-\eqref{L4} we are able to apply Lemma \ref{lemma2.1}
and conclude that there exists a universal constant $c>0$ such that
$$
   v^{\varepsilon}(0) \le c.
$$
Moreover, from Harnack inequality
$$
   v^{\varepsilon} \le C_0 \quad \mbox{in} \quad B_{1/2}.
$$
Therefore, by $C^{1,\alpha}$ regularity estimates (see for example \cite{CC}) we must have that
$$
	|\nabla u^{\varepsilon}(X)| = |\nabla v^{\varepsilon}(0)| \le \frac{1}{\delta_{\varepsilon}} \|u^{\varepsilon} - \varepsilon\| \le C_0,
$$
and the Lemma is proved.
\end{proof}

\begin{lemma}\label{prop2.1}
For $X \in B^{+}_{\frac{1}{8}} \cap \{u^{\varepsilon} > \varepsilon\}$ such that $\delta(X) < \delta_{\varepsilon}(X) \le4 \delta(X)$, we have
$$
	|\nabla u^{\epsilon}(X)| \le C_0
$$
for some constant $C_0 = C_0(n,\lambda,\Lambda, b, \mathcal{A}, \mathcal{B}) >0$.
\end{lemma}
\begin{proof}
Similar to Lemma \ref{prop2}, we may assume that $\delta_{\epsilon} \le \frac{1}{8}$,
otherwise, as in Lemma \ref{prop2} the gradient boundedness follows from local estimates \cite{RT, Tei}.
Define the scaled function $v^{\epsilon} \colon B_1 \rightarrow \mathbb{R}$ by
$$
	v^{\epsilon}(Y) \defeq \frac{u^{\epsilon}(X+\delta Y)-\epsilon}{\delta},
$$
where $\delta = \delta(X)$. Clearly
$$
   F_{\delta}(Y, \nabla v^{\varepsilon}, D^2 v^{\epsilon})=0 \quad \mbox{in} \quad B_1
$$
in the viscosity sense, and, from Harnack inequality
$$
	v^{\varepsilon} \le C v^{\varepsilon}(0) \sim \frac{1}{\delta} \quad \textrm{in} \,\,\, B_{\frac{1}{2}}.
$$
By applying once more $C^{1,\alpha}$ regularity estimates, we obtain
\begin{equation}\label{inf}
	|\nabla u^{\varepsilon}(X)| = |\nabla v^{\varepsilon}(0)| \le \frac{C}{\delta}.
\end{equation}

Therefore, the idea is to find an estimate for $u^{\varepsilon}-\varepsilon$ in terms of the vertical
distance $\delta(X).$ To this end, consider $h$ the viscosity solution to the
Dirichlet problem
\begin{equation} \label{(14)}
\left\{
\begin{array}{rclcl}
  F(X,\nabla h,D^2 h)& = & 0 & \mbox{in} & B^{+}_{1}\\
  h & = & u^{\epsilon} & \mbox{on} &  \partial B^{+}_{1} .
\end{array}
\right.
\end{equation}
Since $0 \le u^{\epsilon} \le C(n, \lambda, \Lambda, b, \mathcal{B})$, it follows from $C^{1,\alpha}$ estimate up to boundary \cite{Wint} that $h \in C^{1,\alpha}\left(\overline{B}^{+}_{\frac{3}{4}}\right)$. Moreover
$$
	|\nabla h(X)| \le \overline{C} \left(\|h\|_{L^{\infty}(B^{+}_1)} + \|\varphi\|_{C^{1,\gamma}( B^{\prime}_{1})}\right) \le \overline{C}(C+\mathcal{A})\defeq \mathcal{C}^{\ast}.
$$
From Comparison Principle, we have that
$$
   u^{\epsilon} \le h \quad \mbox{in} \quad B^{+}_{1}.
$$
Hence,
\begin{equation}\label{(15)}
	u^{\epsilon}(X) \le h(X) \le h(\hat{X}) + \mathcal{C}^{\ast}|X-\hat{X}| \le \varphi(\hat{X}) + \mathcal{C}^{\ast}\delta.
\end{equation}
Now, we have  that $|\hat{X}| \le |X| + \delta \le \frac{1}{4}$ , and, consequently we are able to apply Lemma  \ref{l5} which gives
\begin{equation}\label{(16)}
	\varphi(\hat{X}) \le \epsilon + c_0 \cdot \textrm{dist}(\hat{X}, \{0 \le u^{\varepsilon} \le \varepsilon\}) \le \epsilon + c_0(\delta_{\epsilon} + \delta) \le \epsilon + 5 c_0 \delta.
\end{equation}
Thus, it follows from \eqref{(15)} and \eqref{(16)} that
$$
   u^{\epsilon}(X) - \varepsilon \le C_0 \delta,
$$
where $C_0 \defeq C(5c_0+\mathcal{C}^{\ast})$. Finally, if we apply $C^{1,\alpha}$ estimate, Harnack inequality and estimate \eqref{inf},
respectively, we end up with
$$
	|\nabla u^{\varepsilon}(X)| = |\nabla v^{\varepsilon}(0)| \le \frac{1}{\delta} \|u^{\varepsilon} - \varepsilon\|_{L^{\infty}\left(B_{\frac{1}{2}}\right)} \le C_0
$$
which concludes the proof.
\end{proof}

\begin{lemma}\label{prop2.2}
If $X \in B^{+}_{\frac{1}{8}} \cap \{u^{\varepsilon} > \varepsilon\}$ and $4\delta(X) < \delta_{\varepsilon}(X) $, then there exists a constant $C_0 = C_0(n, \lambda,\Lambda, b, \mathcal{A}, \mathcal{B}) >0$ such that
$$
	|\nabla u^{\epsilon}(X)| \le C_0.
$$
\end{lemma}
\begin{proof}
Initially we will consider the case when $\delta_{\epsilon} \le \frac{1}{8}$. The following inclusion holds true:
$B^{+}_{\frac{\delta_{\epsilon}}{2}}(\hat{X}) \subset B^{+}_{\frac{1}{4}} \setminus \{0 \le u^{\varepsilon} \le \varepsilon\}$.
In fact, if $Y \in B^{+}_{\frac{\delta_{\varepsilon}}{2}}(\hat{X})$ then
$$
	|Y| \le |Y-X| + |X| \le 2 \frac{\delta_{\epsilon}}{2} + |X| \le \frac{1}{4}.
$$
Now, using the same argument as in Lemma \ref{prop2.1} (see \eqref{(14)}) we are able to estimate
$u^{\epsilon}$ in $B^{+}_{\frac{\delta_{\epsilon}}{2}}(\hat{X})$ as follows
$$
	u^{\epsilon}(Y) \le u^{\epsilon}(\hat{Y}) + \mathcal{C}^{\ast} \frac{\delta_{\epsilon}}{2} \le \epsilon + c_0 \cdot \textrm{dist}(\hat{Y}, \{0 \le u^{\varepsilon} \le \varepsilon\}) + \mathcal{C}^{\ast} \frac{\delta_{\epsilon}}{2}.
$$
Since the distance function is Lipschitz continuous with Lipschitz constant $1$, we have
$$
	\textrm{dist}(\hat{Y}, \{0 \le u^{\varepsilon} \le \varepsilon\}) \le \delta_{\epsilon} + |\hat{Y}-X| \le 2\delta_{\epsilon}.
$$
Therefore,
$$
	u^{\epsilon}(Y) \le \epsilon + \left(2c_0 + \frac{\mathcal{C}^{\ast}}{2}\right) \delta_{\epsilon} = \epsilon + c \delta_{\epsilon}.
$$
By considering the function $v^{\epsilon}(Y) = u^{\epsilon}(Y) - \epsilon$ in $B^{+}_{\frac{\delta_{\epsilon}}{2}}(\hat{X})$,
we have that
$$
	F(Y, \nabla v^{\varepsilon}, D^2 v^{\epsilon}) =0 \quad \textrm{in} \quad B^{+}_{\frac{\delta_{\epsilon}}{2}}(\hat{X})
$$
in the viscosity sense. From $C^{1,\alpha}$ estimate up to boundary and Lemma \ref{ABP}, we have
$$
	|\nabla u^{\epsilon}(X)| = |\nabla v^{\epsilon}(X)| \le C (c+\mathcal{A}).
$$

On the other hand, for the case $\delta_{\epsilon} \ge \frac{1}{8}$  we have the following inclusion
$B^{+}_{\frac{1}{16}}(\hat{X}) \subset B_1 \setminus \{0 \leq u^{\varepsilon} \le \varepsilon\}$. In this situation,
since $\textrm{supp}(\zeta_{\epsilon})=[0,\epsilon]$,
$$
\left\{
\begin{array}{rcl}
 F(X, \nabla u^{\varepsilon}, D^2 u^{\epsilon}) = 0 & \mbox{in} & B^{+}_{\frac{1}{16}}(\hat{X})\\
0 \le u^{\epsilon}  = \varphi \le C & \mbox{on} & B^{\prime}_{\frac{1}{16}}(\hat{X})
\end{array}
\right.
$$
and, consequently, the estimate will follow from $C^{1,\alpha}$ estimates up to the boundary.
\end{proof}

%
%
An immediate consequence of Theorem \ref{principal2} and Lemma \ref{ABP}is the existence of solutions via compactness in the Lip-Topology for any family $(u^{\varepsilon})_{\varepsilon >0}$ of viscosity solutions to singular perturbation problem \eqref{Equation Pe}. We consequently obtain

\begin{theorem}[{\bf Limiting free boundary problem}]\label{limFB}
Let $(u^{\varepsilon})_{\varepsilon >0}$ be a family of solutions to \eqref{Equation Pe}.
For every $\varepsilon_{k} \to 0^{+}$ there exist a subsequence $\varepsilon_{k_j} \to 0^{+}$ and $u_0 \in C^{0, 1}(\overline{\Omega})$ such that
\begin{enumerate}
\item[{\bf (1)}] $u^{\varepsilon_{k_j}} \to u_0$ uniformly in $\overline{\Omega}$.
\item[{\bf (2)}] $F(X,\nabla u_0, D^2 u_0)=0$ in $\overline{\Omega} \cap \{u_0>0\}$ in the viscosity sense.
\end{enumerate}
\end{theorem}

\section{Appendix}\label{Append}

\hspace{0.6cm}In this final section we are going to give the proof of some technical results,
which were temporarily omitted.

\begin{lemma}[{\bf Boundary's estimates propagation Lemma}]\label{lemma2.0}
Suppose that $u \ge 0$ is a viscosity solution to
$$
\left\{
\begin{array}{rcl}
	F(X, \nabla u, D^2 u) = 0 & \mbox{in} & B^{+}_{1}\\
	u \ge \sigma >0 & \mbox{on} & B^{\prime}_{1}.
\end{array}
\right.
$$
 Then there exists a universal constant $C= C(n,\lambda,\Lambda, b) >0$ such that
$$
	u(X) \ge C \sigma, \quad X \in B^{+}_{\frac{3}{4}}.
$$
\end{lemma}

\begin{proof}
First of all consider the following Dirichlet problem
\begin{equation} \label{D1}
\left\{
\begin{array}{rcccl}
	F(X, \nabla w, D^2 w) & = &  0 & \mbox{in} & B^{+}_{1}\\
 w & = & \sigma & \mbox{on} & B^{\prime}_1\\
w & = & 0 & \mbox{on} & \partial B_{1} \cap \{X_n >0\}.
\end{array}
\right.
\end{equation}
From $C^{1,\alpha}$ regularity estimate, \cite[Theorem 3.1]{Wint}  we have $w \in C^{1,\alpha}\left(\overline{B}^{+}_{\frac{3}{4}}\right)$,
and, by the Comparison Principle
\begin{equation} \label{D3}
	0 \le w \le \sigma \quad \textrm{in} \,\,\, B^{+}_{1}.
\end{equation}
From now on, it is appropriate we define the following reflection $\mathfrak{U} \colon B_1 \rightarrow \mathbb{R}$,
\begin{equation} \label{D2}
\mathfrak{U}(X) \defeq \left\{
\begin{array}{ccl}
	w(X) & \mbox{if} & X \in B^{+}_{1} \cup B^{\prime}_1\\
	2 \sigma - w(X_1, \ldots,X_{n-1} -X_n) & \mbox{if} & X \in  B_{1} \cap \{X_n <0\}.
\end{array}
\right.
\end{equation}
We observe that $\mathfrak{U}$ is a viscosity solution to
$$
	\mathcal{G}(X, \nabla \mathfrak{U}, D^2 \mathfrak{U})=0 \quad \textrm{in} \quad B_1,
$$
where
$$
\mathcal{G}(X, \overrightarrow{p}, M) \defeq \left\{
\begin{array}{rcl}
	F(X, \overrightarrow{p}, M) & \mbox{if} & X_n \geq 0\\
	-F(\widetilde{X}, \overrightarrow{\widetilde{p}}, \widetilde{M}) & \mbox{if} & X_n<0,
\end{array}
\right.
$$
with
$$
\begin{array}{rcl}
	\widetilde{X} & \defeq & (X_1, \ldots , X_{n-1}, -X_n),\\

    \widetilde{p} & \defeq & (-p_1, \ldots, -p_{n-1}, p_n),\\

    \widetilde{M} & \defeq & \left\{
\begin{array}{rcl}
	-M_{ij} & \mbox{if} & 1 \leq i, j \leq n-1 \,\, \mbox{or} \,\, i=j=n\\
	 M_{ij} & \mbox{otherwise.} &
\end{array}
\right.
\end{array}
$$
Thus, from \eqref{D3},
$$
\sigma \le \mathfrak{U} \le 2 \sigma \quad \mbox{in} \quad B_{1}^{-}
$$
Hence,
$$
0 \le \mathfrak{U} \le 2 \sigma \quad \mbox{in} \quad B_1.
$$
Moreover, from Harnack inequality we have that
$$
	\sup_{B_{3/4}} \mathfrak{U} \le c_0 \inf_{B_{3/4}} \mathfrak{U}.
$$
Particularly,
$$
w(X) \ge c^{-1}_{0} \sigma \quad in \quad B^{+}_{\frac{3}{4}}.
$$
Therefore, the proof follows through the previous inequality combined with the Comparison Principle.
\end{proof}

\begin{lemma}[{\bf Hopf's type boundary principle}]\label{lemma2.1}
Let  $u$ be a viscosity solution to
$$
\left\{
\begin{array}{rcl}
	F(X, \nabla u, D^2 u) = 0 & \mbox{in} & B_{r}(Z)\\
 u \ge 0 & \mbox{in} & B_{r}(Z).
\end{array}
\right.
$$
with $r \le 1$. Assume that for some $X_0 \in \partial B_r(Z)$,
$$
	u(X_0)=0 \quad \textrm{and} \quad \frac{\partial u}{\partial \nu}(X_0) \le \theta,
$$
where $\nu$ is the inward normal direction at $X_0$. Then there exists a universal constant
$C >0$ such that
$$
	u(Z) \le C\theta r.
$$
\end{lemma}

\begin{proof}
By using a scaling argument, we may assume $r=1$. Indeed, it is sufficient to consider the scaled function $v : B_1 \rightarrow \mathbb{R}$
$$
 v_r(Y) = \frac{u(Z+r Y)}{r}.
$$
As before, $v_r$ is a viscosity solution of
$$
	F_r(Y,\nabla v_r, D^2 v_r) =0 \quad \textrm{in} \quad  B_1,
$$
with
$$
   F_r(Y, \overrightarrow{p}, M) \defeq rF\left(Z+rY, \overrightarrow{p}, \frac{1}{r} M\right)
$$
Let $\mathfrak{A} \defeq B_{1} \setminus B_{\frac{1}{2}}$ be an annular region and define $\omega \colon \overline{\mathfrak{A}} \rightarrow \mathbb{R}$ by
$$
	\omega(Y) \defeq \mu \left(e^{-\delta |Y|^2} - e^{-\delta}\right)
$$
where the positive constants $\mu$ and $\delta$ will be chosen \textit{a posteriori}. One can computer the gradient and Hessian of $\omega$ in $\mathfrak{A}$ as follows
\begin{eqnarray*}
	\partial_{i} \omega(Y) &=& -2 \mu \delta Y_i e^{-\delta |Y|^2},\\
	\partial_{ij} \omega(Y) &=& 4 \mu \delta^2 Y_i Y_j e^{-\delta |Y|^2} -2 \mu \delta e^{-\delta |Y|^2} \delta_{ij},\\
    |\nabla \omega(Y)| &= & 2\mu \delta e^{-\delta |Y|^2}|Y|.
\end{eqnarray*}
In particular, for every $M \in \mathcal{A}_{\lambda,\Lambda}\defeq \left\{A \in \textrm{Sym}(n) \suchthat \lambda \|\xi\|^2 \le \sum\limits_{i,j=1}^{n}A_{ij} \xi_i \xi_j \le \Lambda \|\xi\|^{2}, \, \forall \, \xi \in \mathbb{R}^n\right\}$ we have
\begin{eqnarray*}
	\textrm{Tr} \left(M \cdot D^2 \omega\right) - b |\nabla \omega| &=& \sum_{i,j=1}^{n} m_{ij} \partial_{ij} \omega - b \cdot \sqrt{\sum_{i=1}^{n} (\partial_{i} \omega)^2}  \\
	&=& 4 \mu \delta^2 e^{-\delta |Y|^2} \textrm{Tr}(M \cdot Y \otimes Y) - 2 \delta \mu \textrm{Tr}(M) e^{-\delta |Y|^2} - 2 \mu \delta b |Y| e^{-\delta |Y|^2}\\
	&\ge& 4 \mu \delta^2 \lambda |Y|^2 e^{-\mu |Y|^2} - 2 \delta \mu n \Lambda e^{-\delta |Y|^2} - 2 \mu \delta b |Y| e^{-\delta |Y|^2}\\
	&=& 2 \mu \delta (2 \delta \lambda |Y|^2 - b|Y| - n  \Lambda) e^{-\delta |Y|^2} \\
	&\ge& 2 \mu \delta \left(\frac{\delta \lambda}{2} - b - n \Lambda\right) e^{-\delta |Y|^2} \quad \textrm{in} \quad \mathfrak{A},
\end{eqnarray*}
where $\xi \otimes \xi = (\xi_i \xi_j)_{i,j}$.  Choose and fix $\delta \ge \frac{2}{\lambda} (b + n \Lambda)$.
Then, it follows readily that
$$
	\mathcal{P}_{\lambda, \Lambda}^{-} (D^2 \omega) - b |\nabla \omega| \ge  0 \quad \textrm{in} \quad \mathfrak{A}.
$$
Therefore, since $r \le 1$, if $\delta\in \left[\frac{2}{\lambda}(\tilde{b}+ n \Lambda), + \infty\right)$, with $\tilde{b} = r b$, we have
$$
	F_r(Y, \nabla \omega(Y), D^ 2\omega(Y)) \ge 0 \quad \textrm{in} \quad \mathfrak{A}.
$$
Now by Harnack inequality
$$
	v_r(0) \le \sup_{B_{1/2}} v_r \le c_0 \inf_{B_{1/2}} v_r,
$$
Hence
$$
      v_r(Y) \ge c_o^{-1}v_r(0) \quad \mbox{in} \quad B_{\frac{1}{2}}.
$$
By choosing $\mu =  \frac{v_r(0)}{c_0\left(e^{-\frac{\delta}{4}} - e^{-\delta}\right)}$ we have
$$
	\omega \le v_r \quad \mbox{on} \quad \partial \mathfrak{A}
$$
and Comparison Principle gives that
$$
   \omega \le v_r  \quad \mbox{in} \quad \mathfrak{A}
$$
Thus, if we label $Y_0 \defeq  \frac{X_0-Z}{r}$ then
$$
	\mu\delta e^{- \delta} \le \frac{\partial \omega}{\partial \nu}(Y_0) \le \frac{\partial v_r}{\partial \nu}(Y_0) \leq \theta.
$$
Therefore,
$$
	v_r(0) \le \theta \delta^{-1} c_0 \left(e^{\frac{3\delta}{4}}-1\right),
$$
and by returning to the original sentence we can conclude that
$$
  u(Z) \le c \theta r.
$$
\end{proof}

\section*{Acknowledgements}
\addcontentsline{toc}{section}{Acknowledgements}

\hspace{0.56cm} The authors would like to thank Eduardo V. Teixeira for insightful comments and suggestions that benefited a lot the final outcome of this article. We would like to thank anonymous referee by the careful reading and suggestions throughout this paper.

This article is part of the second author's PhD thesis which would like to thank to Department of Mathematics at Universidade Federal do Cear\'{a}-UFC-Brazil for fostering a pleasant and productive scientific atmosphere during the period of his PhD program. This work has received financial support from CAPES-Brazil, CNPq-Brazil and FUNCAP-Cear\'{a}.

\end{document}